\RequirePackage{fix-cm}
\documentclass[smallcondensed]{svjour3}     
\smartqed  
\usepackage{graphicx}
\usepackage{enumitem}
\usepackage{xcolor}
\usepackage[normalem]{ulem}

\usepackage[utf8]{inputenc}
\usepackage[american]{babel}
\usepackage{amsmath,amssymb}
\usepackage{enumitem}
\usepackage{nicefrac}
\usepackage{multirow}
\usepackage{xcolor}
\newcommand{\F}{\mathbb{F}_q}
\newcommand{\EF}{\mathbb{F}_{q^n}}
\newcommand{\EFs}{\mathbb{F}_{q^{2n}}}
\newcommand{\al}{\alpha}
\newcommand{\be}{\beta}

\newcommand{\change}{\textcolor[rgb]{0.3,0.2,0.9}}
\newtheorem{construction}{Construction}
\begin{document} 
	
\title{A recurrent construction of irreducible polynomials of fixed degree over finite fields}

\author{Gohar M. Kyureghyan \and \mbox{Melsik K. Kyureghyan}}

\institute{  Gohar M. Kyureghyan \at
	University of Rostock,Ulmenstraße 69, Haus 3, 18057 Rostock, Germany \\
	 \email{gohar.kyureghyan@uni-rostock.de}         
	\and
	Melsik K. Kyureghyan \at
	Institute for Informatics and Automation Problems, P.\,Sevak Str. 1, 0014 Yerevan, Armenia \\
	\email{melsik@ipia.sci.am}
}

\date{Received: date / Accepted: date}

\maketitle
	
\begin{abstract}
	In this paper we consider in detail the composition of an irreducible polynomial with $X^2$ and suggest a recurrent construction of irreducible polynomials of fixed degree over finite fields of odd characteristics. More precisely, given an irreducible polynomial of degree $n$
	and order $2^rt$ with $t$ odd, the construction produces \change{$ord_t(2)/d$\footnote{\change{In the previous version, the first author calculated this 
	 number  incorrectly  as $ord_t(2)$, see Section 2 for corrections put in blue.}}} irreducible polynomials of degree $n$ and order $t$, for a certain divisor $d$ of $n$.
	The construction can be used for example to search irreducible polynomials with specific requirements on its coefficients.
	
\end{abstract}

\section{Introduction} \label{sec:main}

The so-called composition method is a powerful tool to study and construct
  polynomials 
over finite fields. It is  extensively used for construction of irreducible polynomials, computing a square root 
 and factorization of polynomials, see for example \cite{cohen,mkyureg1,mkyureg2,kyureg-kyureg,martinez-reis-silva,tux-wang}. 
 General recurrent constructions of irreducible polynomials based on composition of irreducible polynomials 
with quadratic rational functions
are suggested in \cite{mkyureg1,mkyureg2}. For these constructions it is not yet understood  which of them
 are particularly well suited for algorithmic applications.
In this paper we consider in detail the composition of an irreducible polynomial with $X^2$ and suggest a recurrent construction of irreducible polynomials of fixed degree over finite fields of odd characteristics. 
 \\

Let $q$ be a power of an odd prime number.
In this paper we consider the composition of an irreducible polynomial $A(X)\ne X$ of degree $n\geq 1$
over  the finite field $\F$ with the polynomial $X^2$. Set $B(X) := A(X^2)$.
If an element $\beta$ from an extension field of $\F$ is a zero of $B(X)$,
then $\be^2$ is a zero of $A(X)$. Consequently  $\beta ^2$ is a proper element of $\EF$, that is
$\be^2$ belongs to $\EF$ but not to any  subfield $\mathbb{F}_{q^s} \ne \EF$ of it.  
Next we need to distinguish whether
 $\be^2$ is a non-square or square in $\EF$. In the first case  $\beta \not\in \EF$ while $\beta \in \EFs$, or
equivalently $B(X)$ is irreducible over $\F$. In the case when $\beta^2$ is a square in $\EF$ the element $\beta $ belongs to $\EF$, and clearly it is  proper
in it since $\beta^2$ is so. Hence the minimal polynomial $C(X)$ of $\beta$ over $\F$ has degree $n$ and it
is a factor of $B(X)$. Clearly, along with
$\beta$ also $-\beta$ is a zero of $B(X)$. The polynomial $(-1)^nC(-X)$ is
the minimal polynomial of $-\beta$ over $\F$. In particular, if $C(X) \ne (-1)^nC(-X)$, then the 
polynomial $B(X)$ is the product of monic irreducible polynomials $C(X)$ and $(-1)^nC(-X)$. 
The next well-known lemma describes in more detail the factorization of $B(X)$. It shows in particular, that 
under our assumptions $C(X) \ne (-1)^nC(-X)$ always holds.

\begin{lemma}\label{lem:start}
Let $q$ be odd, $n\geq 1$ and $\al \ne 0$ a proper element of $\EF$.  
If $A(X)  \in \F[X]$ is the minimal polynomial of $\al$, then the polynomial $B(X) := A(X^2)$ has the following properties:

\begin{itemize}
\item[(a)] $B(X)$ is irreducible over $\F$ if and only if $\al$ is a non-square in $\EF$.
\item[(b)] $B(X)$ is irreducible over $\F$ if and only if $(-1)^nA(0)$ is a non-square in $\F$.
\item[(c)] $B(X)$ is the product of two different irreducible polynomials  $C(X)$ and \\
$(-1)^nC(-X)$ over $\F$ if and only if 
$\al$ is a square in $\EF$.
\end{itemize} 
\end{lemma}

\begin{proof}
The statement in (a) follows from the discussions before this lemma. 
To prove (b), 
recall that $(-1)^nA(0)$ is  the norm $N(\al) = \al^{(q^n-1)/(q-1)}$ of $\al$ over $\F$. 
Note that $\al$ is a square in $\EF$ exactly when its norm is a square in $\F$. 
Indeed, let $S_{q^n}$ and $S_q$ be the sets of non-zero squares in $\EF$ resp. in $\F$.
Then the image $N(S_{q^n})$ is a subset of $S_q$. Since the size of the preimage $N^{-1}(S_q)$ is
$(q^n-1)/2 = |S_{q^n}|$, the equality $N(S_{q^n}) = S_q$ holds. Hence using (a)  $B(X)$ is reducible over $\F$ if and only if $\al$ is a square in $\EF$.
By discussions before this lemma, to complete the proof of (c) it remains to show that
$C(X) \ne (-1)^nC(-X)$. Recall that $A(X)$ factorizes over $\EF$ as follows
\begin{equation}\label{eq:min-pol}
A(X) = (X-\al)(X-\al^q) \ldots (X-\al^{q^{n-1}}),
\end{equation}
and consequently
$$
B(X) =A(X^2) = (X^2-\al)(X^2-\al^q) \ldots (X^2-\al^{q^{n-1}}).
$$
Let $\beta \in \EF$ such that $\al=\beta^2$ and then
$$
X^2-\al = (X-\be)(X+\be).
$$
Observe that since $\al$ is a proper element of $\EF$ so is $\be$ too. This
implies that
$$
C(X) := (X-\be)(X-\be^q) \ldots (X-\be^{q^{n-1}})
$$
is the minimal polynomial of $\beta$ over $\F$ and 
$$
(X+\be)(X+\be^q) \ldots (X+\be^{q^{n-1}}) = (-1)^nC(-X)
$$
is the minimal polynomial of $-\be$ over $\F$. Hence
$C(X) \ne (-1)^nC(-X)$ is equivalent to the property that the minimal polynomials
of $\beta$ and $-\beta$ are different. These minimal polynomials 
are equal if and only if $-\beta$ is a conjugate of $\beta$ over $\F$, that is
$\beta^{q^i} = -\beta$ for some $1\leq i \leq n-1$. 
In the latter case $\alpha^{q^i} = \alpha$ and hence $\al \in \mathbb{F}_{q^{\gcd{(n,i)}}}$, a contradiction to the assumption $\al$ is proper in $\EF$.
\qed
\end{proof}

Two immediate consequences of Lemma \ref{lem:start} are the characterization
of the minimal polynomials of proper elements $\be$ in $\mathbb{F}_{q^{2n}}$ with $\be^2 \in \EF$ and 
 a construction
of irreducible polynomials of degree $2^kn$ by blowing up those of degree $n$:

\begin{corollary}\label{cor:minpol}
Let $q$ be odd and $n \geq 1$.
A proper element $\be$  in $\mathbb{F}_{q^{2n}}$ satisfies $\be^2 \in \EF$ if and only
if the minimal polynomial $B(X)$ of $\be$ over $\F$ fulfills $B(X) = A(X^2)$ for some 
$A(X) \in \F[X]$.  In such a case, 
 $A(X)$ is the minimal polynomial of $\be^2$ over $\F$.
\end{corollary}
\begin{proof}
If $B(X)$ is irreducible, then the polynomial $A(X)$ is irreducible as well.
Hence a zero $\al$ of $A(X)$ is a proper element of $\EF$ satisfying $\al = \be^2$.
Suppose now  $\be \in \mathbb{F}_{q^{2n}}$ is proper and $\be^2 \in \EF$.
Then $\be^2$ is a non-square in $\EF$, and hence the statement follows from
Lemma \ref{lem:start}\,(a). \qed
\end{proof}

\begin{corollary}\label{cor:blowup}
Let $q$ be odd and $F(X) \in \F[X]$ be monic and irreducible of degree 
$n\geq 1$ with $(-1)^nF(0)$ a non-square
in $\F$. Then $F(X^2)$ is irreducible over $\F$. The
polynomial $F(X^{2^k})$ with $k\geq 2$ is irreducible over $\F$ if and only if either
 $q \equiv 3 \pmod 4$ and $n$ is even, or $q \equiv 1 \pmod 4$.
\end{corollary}
\begin{proof} 
The irreducibility of $F(X^2)$ follows directly from Lemma \ref{lem:start}\,(b).
Let $k = 2$. Using again \ref{lem:start}\,(b) the polynomial $F(X^4)$ is irreducible
if and only if $(-1)^{2n}F(0) = F(0)$ is a non-square in $\mathbb{F}_q$. The latter is not
fulfilled if and only if $(-1)^n$ is a non-square in $\F$, that is if and only if $n$ is odd and $q \equiv 3 \pmod 4$. It remains to observe
that $F(X^{2^k})$ for $k \geq 3$ is irreducible if and only if $F(X^4)$ is irreducible
over $\F$. \qed
\end{proof}

An important feature of Corollary \ref{cor:blowup} is that it ensures the existence of sparse irreducible polynomials of degree $2^kn$, as the following example demonstrates:

\begin{example}
The polynomial $F(X)=X^6+X+3$ is irreducible over $\mathbb{F}_{19}$ and  $(-1)^6F(0)=3$ is a non-square in $\mathbb{F}_{19}$. By Corollary \ref{cor:blowup}  the 
trinomial $X^{2^k6}+X^{2^k}+3$ is irreducible over $\mathbb{F}_{19}$ for any $k\geq 1$.

\end{example}

Recall that the order of an irreducible polynomial $F(X)\ne X \in \F[X]$ of degree $n$
is defined as the order of its zero in $\EF$. We denote it by $ord(F(X))$.

For our next discussions we need the following observation on the irreducible  polynomials satisfying $C(X) = (-1)^nC(-X)$.

\begin{proposition} \label{prop:even_coef}
Let $q$ be odd, $n\geq 2$ and $C(X) = X^n +\sum_{i=0}^{n-1}c_{i}X^{i}$ $ \in \F[X]$ (we set $c_n=1$) be an
irreducible polynomial. Then the following statements are equivalent:
\begin{itemize}
\item[(a)] $C(X) =  D(X^2)$ for some $D(X) \in \F[X]$, that is $c_i=0$ for all odd indices $1\leq i \leq n$.
\item[(b)] $C(X) = (-1)^nC(-X)$.
\item[(c)] The order of $C(X)$ divides $2(q^{n/2}-1)$.  
\end{itemize}
\end{proposition}
\begin{proof}
The degree $n$ of  $C(X) = D(X^2)$ is even, and hence in such a case $$ (-1)^nC(-X) = D((-X)^2) =D(X^2) = C(X),$$ proving
the implication (b) from (a). Next we show that for an irreducible polynomial $C(X)$ of degree $n\geq 2$ from (b) follows (a).
Suppose $C(X) = (-1)^nC(-X)$. If $n$ is even the considered equality reduces to $C(X) =C(-X)$. The latter is
satisfied if and only if $C(X) = \sum_{i=0}^{n/2}c_{2i}X^{2i}$ or equivalently if
$C(X) = D(X^2)$ for an appropriate polynomial $D(X)$ of degree $n/2$.
For $n$ odd we get $C(X) =-C(-X)$, which forces   $c_{i} =0$ for all even indicies $i$,
in particular $c_0=0$ too. The irreducibility of  $C(X)$ yields then $C(X)=X$.
Hence (a) and (b) are indeed equivalent. Next we show equivalence of (a) and (c).
Let $\alpha$ be a root of $D(X)$ and $\beta $ of $C(X)$.  Then $\be^2 = \al$. Since $\al \in \mathbb{F}_{q^{n/2}}$, from (a) follows (c). Suppose (c) holds and $\be \in \EF$ is a root of $C(X)$.
 Then $\al:= \beta^2$ has order dividing $(q^{n/2}-1)$, and hence 
 $\al$ is  in $\mathbb{F}_{q^{n/2}}$. Further $\alpha$ is proper in $\mathbb{F}_{q^{n/2}}$,
since $\beta$ is proper in $\EF$. This implies that (a) holds with
$D(X)$ being the minimal polynomial of $\alpha$. 
\qed \end{proof}

\hspace*{-0.3cm} The next result is obtained by reversing arguments of Lemma \ref{lem:start} and 
\mbox{Proposition \ref{prop:even_coef}}.

\begin{corollary}\label{cor:prod-irred} Let $q$ be odd and $n \geq 1$.
Let $C(X) = X^n+ \sum_{i=0}^{n-1}c_iX^i$ be a monic irreducible polynomial of degree $n$ over $\F$ (we set here $c_n=1$). Then  there is a polynomial $A(X) \in \F[X]$ of degree $n$ such that
$$ C(X) \cdot (-1)^nC(-X) = A(X^2).$$ 
More precisely,
\begin{equation} \label{form}
A(X) = (-1)^n\sum_{j=0}^n\sum_{u=0}^{2j}(-1)^uc_uc_{2j-u}X^j, \mbox{ with } 
c_s = 0 \mbox{ for } s >n.
\end{equation}
\begin{itemize}\item[(a)] The polynomial $A(X)$ is irreducible over $\F$ if and only if 
there is at least one odd $1 \leq i \leq n$ with $c_i \ne 0$.
\item[(b)] If $A(X)\ne X$ is irreducible then it is the minimal polynomial of
$\beta^2$, where 
 $\beta \in \EF$ is a zero of $C(X)$. In this case $ord(A(X)) = 
ord(C(X))/\gcd(2,ord(C(X)))$. 
\end{itemize}
\end{corollary}

\begin{proof}
Set $F(X) := C(X) \cdot (-1)^nC(-X)$. Since by construction $F(X) = F(-X)$,
there is a polynomial $A(X)$ satisfying $F(X)=A(X^2)$. Direct calculations show that $A(X)$
is given by the formula (\ref{form}). To prove (a),
note that if  $c_i=0$ for  all odd $i$, then $C(X) = D(X^2)$ for a certain
$D(X) \in \F[X]$.  Hence $A(X^2) = (-1)^nC(X)C(-X) = D(X^2)^2$, implying $A(X) = D(X)^2$.  
So it remains to show 
that $A(X)$ is irreducible if there is at least one odd $i$ with $c_i\ne 0$.
Let $C(X)$ be the minimal polynomial of $\beta \in \EF$. Then $A(\be^2)=0$ and
thus minimal polynomial of $\be^2$ divides $A(X)$. Since there is an odd $i$ with $c_i \ne 0$,
 Corollary \ref{cor:minpol} implies that $\be^2$ is not contained in any proper
subfield of $\EF$.  This shows that the minimal polynomial of $\beta^2$ has degree $n$,
and hence $A(X)$ is the minimal polynomial of it. This proves also (b).
\qed
\end{proof}

In next section we use Corollary \ref{cor:prod-irred} to construct irreducible polynomials
from a given one. For this construction also the following easy observation is of interest.

\begin{proposition}\label{prop:x+1}
Let $B(X) = A(X^2) \in \F[X]$ be monic irreducible polynomial of degree $2n \geq 2$ with $\gcd(n,q)=1$ and $q$ odd. Then for any $a \in \F, a \ne 0,$ the polynomial $F(X) = B(X+a)$ is  irreducible over $\F$ and $F(X)$ has at least one coefficient $f_i \ne 0$ with odd $0\leq i <2n$.
\end{proposition}
\begin{proof}
Clearly $F(X) \in \F[X]$ is irreducible. Next we show that the coefficients
of $X^{2n-1}$ in it is $2na$, which is non-zero under our assumptions.
Let 
$$
B(X) = \sum_{i=0}^{n}a_{i}X^{2i}.
$$
Then
\begin{eqnarray*}
F(X) &=& B(X+a) = \sum_{i=0}^{n}a_{i}(X+a)^{2i} \\ &= & (X+a)^{2n} +\sum_{i=0}^{n-1}a_{i}(X+a)^{2i} \\ &=& X^{2n} + 2naX^{2n-1} + \ldots.
\end{eqnarray*}
\qed
\end{proof}

\section{Recurrent construction of irreducible polynomials of fixed degree}
\label{sec:irred}

In this section using Corollary \ref{cor:prod-irred} we describe two recursive constructions of irreducible polynomials of degree $n$ from 
a given irreducible polynomial $C(X)$ of degree $n$. In Construction \ref{con1}, we assume that
the order of the initial polynomial $C(X)$ is known and use it to terminate the construction.
In Construction \ref{con2} the order of the initial polynomial $C(X)$ is supposed to be unknown.
 The number of performed iterations
in Construction \ref{con2} can be then used to compute the order of $C(X)$.\\

For an odd natural number $t$,  we denote by $ord_t(2)$  the order of $2$ modulo $t$.
\begin{construction}\label{con1}
Let $q$ be odd, $C(X) = X^n +\sum_{j=0}^{n-1}c_jX^j \in \F[X], \, C(X) \ne X$ and $c_n=1$, be a given irreducible polynomial of degree $n \geq 1$. 
Further suppose the order $ord(C(X)) = 2^rt$ is known, where $r\geq 0$ and $t\geq 1$ odd.
Given a polynomial $C_i$ with 
$i \geq 0$ set $C_{i+1}(X)$ to denote the polynomial of degree $n$ constructed from $C_{i}(X)$ as follows
\begin{equation}\label{step1}
C_{i+1}(X^2) := (-1)^nC_i(X)C_i(-X).
\end{equation}

Put $C_0(X) = C(X)$. 

If $0\leq i \leq r-1$ and the polynomial $C_{i}$  has at least one non-zero odd coefficient,
then continue with \emph{(\ref{step1})} to construct $C_{i+1}$, otherwise stop.

For $r \leq i \leq r+ord_t(2)-2$ with $t>1$ construct $C_{i+1}$ by \emph{(\ref{step1})}. \qed
\end{construction}

The next theorem  describes the performance and proves the correctness of Constructions \ref{con1}.

\begin{theorem}
Let $C(X)$ be as in Construction \ref{con1}. Then the following holds:
 \begin{itemize}
 \item[(1)] If $n$ is odd or $n$ is even and $t$ does not divide
$q^{n/2}-1$, Construction \ref{con1}
 produces one polynomial of order $2^it$ for each $1\leq i \leq r$ and $ord_t(2)$ \change{\sout{different}}
\footnote{\change{The number of different such polynomials is  $ord_t(2)/d$ for a certain divisor $d\geq 1$ of $n$, see \cite{graner-kyureg}.}} polynomials of order $t$ (including $C_0$).
 \item[(2)] If $n$ is even and  $s$, $1\leq s \leq r-2$, is minimal such that $2^{r-s}t$  divides $2(q^{n/2}-1)$, then Construction \ref{con1} yields 
 one polynomial of order $2^{r-i}t$ for each $0\leq i \leq s$ (including $C_0$) and stops.
 \end{itemize}
 \end{theorem}
 
 \begin{proof}
By Corollary \ref{cor:prod-irred} the order of $C_{i+1}(X)$ is equal to $ord(C_i)/\gcd(ord(C_i(X)),2)$ for any $i\geq 0$.
For $0\leq i \leq r-1$ Construction \ref{con1} terminates after producing $C_{i+1}(X)$ if  $C_{i+1}$ has all its odd coefficients equal to $0$.
 This occurs 
if and only if $n$ is even and $ord(C_{i+1})$ divides $2(q^{n/2}-1)$ 
by Proposition \ref{prop:even_coef}\,(c). Otherwise all produced polynomials have a non-zero odd coefficient
and   the construction will stop
after constructing $ord_t(2)$ polynomials of order $t$.
 \qed
\end{proof}

In the case when the order of initial polynomial $C(X)$ is unknown we modify the stopping condition in
Construction \ref{con1}. The stopping condition in this case relies on the following observation: If $q^n-1 = 2^u w$ with $w$ odd, then clearly
$2^{u+1}$ does not divide order of $C(X)$. Hence after $h \leq u$ steps of construction,
 polynomial $C_h$ will have an odd order $t$, and after \change{at most} further $ord_t(2)$ steps the construction
will produce $C_h$ again. Indeed if $\gamma \in \EF$ is a zero of $C_h$ then $\gamma^{2^i}$ is a zero of $C_{h+i}$
and $\gamma^{2^{ord_t(2)}}=\gamma.$

\begin{construction}\label{con2}
Let $q$ be odd, $C(X) = X^n +\sum_{j=0}^{n-1}c_jX^j \in \F[X], \, C(X) \ne X$ and $c_n=1$, be a given irreducible polynomial of degree $n \geq 1$. Further, let $q^n-1  = 2^uw$ with $u\geq 0$ and an odd $w\geq 1$.\\
 
Set $C_0(X) := C(X)$, if the polynomial $C_{0}$  has at least one non-zero odd coefficient,
continue with \emph{(\ref{step1})} to construct $C_{1}$ otherwise stop.

For $1\leq i \leq u$, if the polynomial $C_{i}$  has at least one non-zero odd coefficient
and $C_{i} \notin \{ C_0, \ldots , C_{i-1}\}$,
continue with \emph{(\ref{step1})} to construct $C_{i+1}$ otherwise stop. 

For $i > u$,
while $C_{i} \notin \{ C_0, \ldots , C_{u}\}$,  construct $C_{i+1}$, otherwise stop.
\end{construction}

\begin{remark}\label{order-con2}
Construction \ref{con2}  provides  information on the order of the initial polynomial $C_0$.
Indeed, suppose the last constructed polynomial is $C_k$ with $k \geq 1.$ 
If $C_k(X) = D(X^2)$ then order of $C_0$ is $2^kv$ with $v$ dividing $2(q^{n/2}-1)$ by Proposition \ref{prop:even_coef}. Otherwise suppose
 $C_k = C_l$ with $l<k$.  Then the order of $C_0$ is $2^l m$ where $m$ is an odd divisor of $q^n-1$ with
\change{$k-l = ord_2(m)/d$ with $d\geq 1$ dividing $n$}. 

This observation can also be adapted for computing the order
of an elemenent $\alpha \in \EF$, provided that its minimal
 polynomial $A(X)$  over $\F$ is known. The minimal polynomial can be computed
using  formula (\ref{eq:min-pol}). However it is in general not efficient since
the computations are done in $\EF$.  An alternative way is 
solving a system of linear equations over $\F$, as suggested in \cite[page\,112]{berlekamp-book}.
\end{remark}

\begin{remark}
Construction \ref{con1} produces \change{$ord_t(2)/d$} different irreducible polynomials of degree
$n$ and order $t$ with $d\geq 1$ dividing $n$. In generic case this is going to be a  large amount of polynomials, since there are good indications that the  average order of $2$
modulo an odd integer is large \cite{kurlberg-pomerance,pomerance-shparlinski}.
 An interesting discussion on the topic can be found in \cite{slides-pomerance}. 
\end{remark}

Corollary \ref{cor:blowup} and Construction 
 \ref{con1} or \ref{con2} can be combined to construct polynomials of degree $2^kn$ from a suitable
polynomial of degree $n$: Let $k\geq 1$. Suppose an irreducible polynomial $U(X)\in \F[X]$ satisfies
the conditions of Corollary \ref{cor:blowup} and 
 $H(X) = U(X^{2^k})$ is irreducible as well. Observe that if $H(X)$  is used  as an 
initial polynomial in Construction \ref{con1} or \ref{con2}, then no new  
irreducible polynomials will be produced, since
all odd coefficients in $H(X)$ are equal to zero.  Instead we can take the irreducible polynomial
$H(x+a)$ with  $a \in \F$, which  by Proposition \ref{prop:x+1}  has a non-zero odd coefficient.
This observation allows effective constructions of irreducible polynomials, which are of particular interest  for
small  $n$ and large $k$. The next example illustrates these ideas for $q=19, n=3$ and $k=2$.

\begin{example} \label{ex:u-h} 
Take $q = 19$. Then $U(X) = X^3+X+1$ is irreducible over $\mathbb{F}_{19}$.
By Corollary \ref{cor:blowup} is also $H(X) := U(X^2)=X^6+X^2+1$ irreducible over 
$\mathbb{F}_{19}$. The order of this polynomial is $1524=2\cdot(19^3-1)/9= (19^6-1)/30870$.  By 
Proposition \ref{prop:x+1}  also 
$C(X):=H(X+1)= X^6 + 6X^5 + 15X^4 + X^3 + 16X^2 + 8X + 3$ is irreducible over $\mathbb{F}_{19}$, which
has an odd  non-zero coefficient. The order of
 $C(X)$ is  $9409176 = (19^6-1)/5 = 2^3\cdot 1176147$. If we initialize Construction
\ref{con1} with $C_0(X) =C(X)$, after 3 iterations we get the polynomial $C_3(X)= X^6 + X^5 + 18X^3 + 2X^2 + 7X + 6$ 
of odd order $9409176/8 = 1176147$. Altogether the construction
 yields
$3 + ord_{1176147}(2) = 885$ irreducible polynomials of degree $6$, and $882$ of them have order $1176147$. 
Among these polynomials 9 are with exactly 5 nonzero coefficients, and 198 are with exactly 6 and
the remaining 678 polynomials have all 7 coefficients non-zero.

If we repeat the construction choosing $C(X)$ to be $H(X+5)$. Then $C_0(X) =C(X)$ is primitive,
that is $ord(C_0(X)) = 19^6-1$. In this case we get 1767 irreducible 
polynomials, among which there are 3 polynomials with exactly 4
non-zero coefficients. 

Table\,1 summarizes our numerical calculations for  $H(X+a)$ with all $a \in \mathbb{F}_{19}$.
Observe that  $H(X+a)$ and $H(X-a)$ can be obtained from each other
by substituting $-X$. Hence their orders are either equal (in this case the order is even) or
 differ by factor 2 (and then one of them is odd).
This explains the similarities in behavior of data in Table 1 for polynomials $H(X+a)$ and $H(X-a)$. 
 Our computations are done with SageMath\footnote{We thank Maurin Graner for her support
with these calculations.}. \qed 

\end{example}

It is interesting to note that in Example \ref{ex:u-h} we start with a polynomial $H(X)$ which has
a  small order and then obtain polynomials $H(X+a)$ with large orders, six of them are even of
largest possible order $q^6-1$, that is they are primitive.
A result of   Davenport (for $q$ prime) and Carlitz (any $q$) states, that for  an irreducible polynomial $F(X)$ of sufficiently large degree $n$ over $\F$ there is always
an element $a \in \F$ such that $F(X+a)$ is primitive. However little is known about
the number of element $a \in \F$ with $F(X+a)$ of a specified order.
 Stephen Cohen generalized this result in several directions, for latest developments in this area see \cite{cohen-kap}. \\

\begin{table} 
\renewcommand{\arraystretch}{1.2}
\begin{minipage}[t]{1.0\textwidth}
\caption{{\bf Numerical results on polynomials of Example \ref{ex:u-h}}. Here $N = 19^6-1$ and $k^l$ in
Weight distribution indicates that the construction produces $l$ polynomials with exactly $k$ non-zero terms. 
}
\label{tab}
\centering
\begin{tabular}{|c|c|c|}
	\hline
	$a \in \mathbb{F}_{19}$ & order of $H(x+a)$ & Weight distribution \\
	\hline
	 1 & N/5 &  $5^9 6^{198}7^{678}$\\
	 2 & N/3 &  $5^{18} 6^{121}7^{452}$\\
	 3  & N/4 &  $4^3 5^{42} 6^{348}7^{1371}$\\
	4 & N/2 &  $4^3 5^{33} 6^{364}7^{1366}$\\
	 5 & N &  $4^3 5^{39} 6^{363}7^{1362}$\\
	 6 & N  &  $ 5^{57} 6^{345}7^{1365}$\\
	 7 & N &  $ 5^{54} 6^{370}7^{1343}$\\
	 8 & N/2  &  $ 4^3 5^{42} 6^{343}7^{1378}$\\
	 9  & N/4 &  $ 5^{27} 6^{385}7^{1353}$\\
	10  & N/8 &  $ 5^{27} 6^{384}7^{1353}$\\
  11 & N/2  &  $ 4^3 5^{42} 6^{343}7^{1378}$\\
	12 & N &  $ 5^{54} 6^{370}7^{1343}$\\
	13 & N  &  $ 5^{57} 6^{345}7^{1365}$\\
	14 & N &  $4^3 5^{39} 6^{363}7^{1362}$\\
	 15 & N/2 &  $4^3 5^{33} 6^{364}7^{1366}$\\
	  16 & N/4 &  $4^3 5^{42} 6^{349}7^{1371}$\\
	 17 & N/3 &  $5^{18} 6^{121}7^{452}$\\
	 18 & N/5 &  $5^9 6^{198}7^{678}$\\
	\hline
\end{tabular}
\end{minipage}\hfill
\end{table}


%
%


\begin{thebibliography}{100}
	%
	%
	\bibitem{berlekamp-book} Elwyn Berlekamp, Algebraic Coding Theory, Worls Scientific Publ. Co. Pte. Ltd. (2015).
	
	\bibitem{carlitz}  Leonard Carlitz: 
	Distribution of primitive roots in a finite field, Quart. J. Math. Oxford Ser. (2) 4(1), 4–10 (1953).
	
    \bibitem{cohen} Stephen D. Cohen: The explicit construction of irreducible polynomials over finite fields,
     Des. Codes Cryptogr. 2, 169–173 (1992)	
	
	\bibitem{cohen-kap} Stephen D. Cohen and Giorgos Kapetanakis: Finite field extensions with the line or translate property for $r$-primitive elements, arXiv:1906.08046  

\bibitem{davenport} Harold Davenport: On primitive roots in finite fields, Quart. J. Math. Oxford 8(1), 
    308–312 (1937).
		
\bibitem{graner-kyureg} Anna-Maurin Graner and Gohar Kyureghyan: ???, in preparation, 2020.
    
	\bibitem{kurlberg-pomerance} P\"ar Kurlberg and Carl Pomerance:
	On a problem of Arnold: The average multiplicative order of a given integer,
	    Algebra Number Theory, 
    Volume 7(4), 981-999 (2013).
	
	\bibitem{mkyureg1} Kyuregyan Melsik: Recurrent methods for constructing irreducible
	polynomials over $\F$ of odd characteristics, Finite Fields Appl. 9(1), 39-58 (2003).
	
\bibitem{mkyureg2} Kyuregyan Melsik: Recurrent methods for constructing irreducible
	polynomials over $\F$ of odd characteristics, II, Finite Fields Appl. 12(3), 357-378 (2006).

\bibitem{kyureg-kyureg}  Kyuregyan Melsik and Gohar Kyureghyan: Irreducible compositions of polynomials over finite fields,
Des. Codes Cryptogr. 61(3), 301-314 (2011).

	
\bibitem{martinez-reis-silva} F.E. Brochero Mart\`inez, Lucas Reis and Lays Silva:	Factorization of composed polynomials and applications, 	
arXiv:1901.02951
	
\bibitem{slides-pomerance} 	Carl Pomerance, \mbox{The multiplicative order mod $n$,
on average}, \\ {https://math.dartmouth.edu/$\sim$carlp/PDF/ordertalk.pdf}	
	
\bibitem{pomerance-shparlinski} Carl Pomerance and Igor E. Shparlinski:
Smooth orders and cryptographic applications, 
C. Fieker and D.R. Kohel (Eds.): ANTS 2002, LNCS 2369, pp. 338–348, 2002

\bibitem{tux-wang} Aleksandr Tuxanidy and Qiang Wang: Composed products and factors of cyclotomic polynomials over finite fields, Des. Codes Cryptogr. 69, 203-231 (2013).

	 
\end{thebibliography}


\end{document}